\documentclass[12pt,twoside,a4paper]{amsart}

\usepackage[utf8]{inputenc}
\usepackage{amssymb, amsthm, amsmath}
\usepackage{amsxtra}
\usepackage{tabularx}
\usepackage{graphicx}
\usepackage{enumerate}

\usepackage{leading}
\leading{18pt}

\usepackage[colorlinks,citecolor=blue,urlcolor=blue]{hyperref}
\usepackage{color}

\newcommand{\ud}{\mathrm{d}}
\newcommand{\ue}{\mathrm{e}}
\newcommand{\supp}{\mathrm{supp}}
\DeclareMathOperator{\tr}{tr}
\DeclareMathOperator{\discr}{discr}
\newcommand{\Id}{\operatorname{Id}}

\newtheorem{tw}{Theorem}
\newtheorem{wn}{Corollary}
\newtheorem{hipoteza}{Conjecture}

\theoremstyle{remark}

\newcommand{\calV}{\mathcal{V}}
\newcommand{\NN}{\mathbb{N}}
\newcommand{\RR}{\mathbb{R}}

\theoremstyle{remark}
\newtheorem{przyklad}{Example}

\hypersetup{
pdfpagemode=UseNone,
pdfdisplaydoctitle=true,
pdfborder={0 0 0}, 
pdftitle={Periodic perturbations of unbounded Jacobi matrices II: Formulas for density},
pdfauthor={Grzegorz Świderski},
pdflang=en-GB
}

\author{Grzegorz Świderski}
\address{
	Grzegorz Świderski\\
	Instytut Matematyczny\\
	Uniwersytet Wrocławski\\
	Pl. Grunwaldzki 2/4\\
	50-384 Wrocław\\
	Poland}
\email{gswider@math.uni.wroc.pl}

\title[Periodic perturbations of unbounded Jacobi matrices II]{Periodic perturbations of unbounded Jacobi matrices II: Formulas for density}

\keywords{Continuous positive density, Turán determinants, Christoffel functions, asymptotics of orthonormal polynomials.}

\subjclass[2010]{Primary: 42C05.}

\begin{document}
   \begin{abstract}
      We give formulas for the density of the measure of orthogonality
      for orthonormal polynomials with unbounded recurrence coefficients.
      The formulas involve limits of appropriately scaled Turán determinants
      or Christoffel functions. Exact asymptotics of the polynomials and 
      numerical examples are also provided.
   \end{abstract}   
   
   \maketitle
   
   \section{Introduction}
      Consider a sequence $(p_n : n \in \NN)$ of polynomials defined by
      \begin{equation}
         \begin{gathered} \label{eq:definicjaWielomianow}
            p_{-1}(x) = 0, \quad p_0(x) = 1, \\
            a_{n-1} p_{n-1}(x) + b_n p_n(x) + a_n p_{n+1}(x) = x p_n(x) \quad (n \geq 0)
         \end{gathered}
      \end{equation}
      for sequences $a = (a_n \colon n \in \mathbb{N})$ and $b = (b_n \colon n \in \mathbb{N})$ satisfying $a_n > 0$ and $b_n \in \RR$. The sequence \eqref{eq:definicjaWielomianow} is orthonormal in $L^2(\mu)$ for a Borel measure $\mu$ on the real line. We are interested in the case when the sequence $a$ is unbounded and the measure $\mu$ is unique. When it holds, we want to find conditions on the sequences $a$ and $b$ assuring absolute continuity of $\mu$ and a constructive formula for its density.
      
      In the case when the sequences $a$ and $b$ are bounded, there are several approaches to an approximation of the density of $\mu$. One is obtained by means of $N$-\emph{shifted Turán determinants}, i.e. expressions of the form 
      \[
         D^N_n(x) = p_n(x) p_{n+N-1}(x) - p_{n-1}(x) p_{n+N}(x)
      \]
      for positive $N$ (see \cite{Nevai79, JGVA, WVA}). Another by \emph{Christoffel functions}, i.e.
      \[
         \lambda_n(x) = \bigg[ \sum_{k=0}^{n} p_k^2(x) \bigg]^{-1}
      \]
      (see \cite{MNT91, VT2009}).
      
      In the unbounded case there is a vast literature concerning qualitative properties of $\mu$ such as: its support, absolute continuity or continuity of $\mu$, localization of its discrete part, see, e.g. \cite{JN3, JD4, DJMP, JM4, JN2, JN1, GS1}. As far as the approximation of $\mu$ is concerned, the only result known to the author is \cite{AptekarevGeronimo2016}. In Section \ref{sec:TuranAndTheMeasure} we prove the following theorem.
      
      \begin{tw}[Regular case] \label{tw:przypadekRegularny}
         Let $N$ be a positive integer. Suppose that
         \begin{equation} \label{eq:7}
            \calV_N\bigg(\frac{1}{a_n} : n\in \NN\bigg) + \calV_N\bigg(\frac{b_n}{a_n} : n \in \NN\bigg) + \calV_1\bigg(\frac{a_{n+N}}{a_n} : n \in \NN\bigg) < \infty.
         \end{equation}
         Let
         \begin{enumerate}[(a)]
            \item 
            $\begin{aligned}
               \sum_{n=0}^\infty \frac{1}{a_n} = \infty,
            \end{aligned}$ \label{tw:1:1}
            
            \item
            $\begin{aligned}
               \lim_{n \rightarrow \infty} \frac{1}{a_n} = 0, \quad \lim_{n \rightarrow \infty} \frac{b_{nN+i}}{a_{nN+i}} = q_i, \quad \lim_{n \rightarrow \infty} \frac{a_{nN+i-1}}{a_{nN+i}} = r_i > 0,
            \end{aligned}$ \label{tw:1:2}
         \end{enumerate}
         and
         \begin{equation} \label{eq:regWarNaDiscr}
            \mathcal{F} = \prod_{j = 0}^{N-1} 
            \begin{pmatrix}
               0 & 1 \\
               -r_j & -q_j
            \end{pmatrix}, \quad \discr{\mathcal{F}} < 0.
         \end{equation}
         Then for each $x \in \RR$ the following limit 
         \begin{equation} \label{eq:zbieznoscTurana}
            g(x) = \lim_{n \rightarrow \infty} a_{n+N-1} D^N_n(x)
         \end{equation}
         exists and defines a continuous function without zeros. Moreover, the measure $\mu$ is absolutely continuous on $\RR$ and its density can be expressed as follows
         \[
            \mu'(x) = \frac{\sqrt{-\discr{\mathcal{F}}}}{2 \pi |g(x)|}, \quad x \in \RR.
         \]
      \end{tw}
      
      Recall that for a matrix $X \in M_2(\mathbb{R})$ its discriminant is defined by $\discr{X} = (\tr{X})^2 - 4 \det{X}$. 
      Let us also recall that the total $N$-variation of a sequence $(x_n : n \in \NN)$ is defined by
      \[
         \calV_N(x_n : n\in \NN) = \sum_{n = 0}^\infty |x_{n+N} - x_n|.
      \]
      Notice that $\calV_N(x_n : n \in \NN) < \infty$ implies that for each $j \in \{0, \ldots, N-1\}$ the subsequence $(x_{k N + j} : k \in \NN)$ converges.
      
      The assumptions of \cite{AptekarevGeronimo2016} correspond to the case $N=1$ of Theorem~\ref{tw:przypadekRegularny}.
      Let us comment on some of the differences between Theorem~\ref{tw:przypadekRegularny} and the results from \cite{AptekarevGeronimo2016}. Our expression for $\mu'$ is different, namely we have characterisation in terms of Turán determinats instead of an infinite product. Furthermore, we do not use complex analysis in our proofs. But more importantly, one can take $N>1$, which allows to cover many perturbations considered in the literature (see \cite[Section 4.3]{GSBT} for details).
      
      The next theorem concerns a critical case of the previous one, i.e. when $\discr{\mathcal{F}} = 0$. This case is particularly interesting in applications (e.g. in the so-called one-quarter class of orthogonal polynomials, see \cite[Section 5]{GS1}).
      
      \begin{tw}[Critical case] \label{tw:przypadekKrytyczny}
         Let $N$ be a positive integer and let
         \[
            q \in \bigg\{2 \cos \frac{\pi}{N}, 2 \cos \frac{2\pi}{N}, \ldots, 2 \cos\frac{(N-1)\pi}{N} \bigg\}.
         \]
         Suppose that
         \[
            \calV_N\left( \frac{1}{a_n} : n \in \NN\right) + \calV_N\left( a_n - a_{n-1} : n \in \NN \right) + \calV_N\big(b_n - q a_n : n \in \NN\big) < \infty.\\
         \]
         Let
         \begin{enumerate}[(a)]
            \item
            $\begin{aligned} 
               \lim_{n \rightarrow \infty} \frac{1}{a_n} = 0, \quad \lim_{k \rightarrow \infty} (a_{N k + i} - a_{N k + i - 1}) = s_i, \quad \lim_{n \rightarrow \infty} (b_n - q a_n) = 0,
            \end{aligned}$ \label{tw:2:1}

            \item
            $\begin{aligned} 
               \lim_{n \rightarrow \infty} (a_{n+N} - a_n) = 0.
            \end{aligned}$ \label{tw:2:2}           
         \end{enumerate}
         Let $x_- \leq x_+$ be the roots of the following polynomial
         \[
            h(x) = x^2 \frac{4 N^2}{4 - q^2} - 4 \sum_{i,j=1}^{N-1} s_i s_j w_{i-1}(q) w_{j-1}(q) w_{i-j}(q),
         \]
         where $w_n$ is the sequence of Chebyshev polynomials of the second kind defined in \eqref{eq:defCzebyszewII}. Then for each $x \in \RR \backslash [x_-,x_+]$ the limit
         \[
            \widetilde{g}(x) = \lim_{n \rightarrow \infty} a^2_{n+N-1} D^N_n(x)
         \]
         exists and defines a continuous function without zeros. Moreover, the measure $\mu$ is absolutely continuous on $\RR \backslash [x_-,x_+]$ and its density can be expressed as follows
         \[
            \mu'(x) = \frac{\sqrt{h(x)}}{2 \pi |\widetilde{g}(x)|}, \quad x \in \RR \backslash [x_-,x_+].
         \]
      \end{tw}
      
      There are known examples when in Theorem~\ref{tw:przypadekKrytyczny} the set of the limit points of the support of $\mu$ equals $\RR \backslash (x_-,x_+)$ (see \cite{DJMP}).
      
      The article is organised as follows. In Section~\ref{sec:Preliminaries} we present our notation and basic formulas. 
      In Section~\ref{sec:TuranAndTheMeasure} we show proofs of Theorem~\ref{tw:przypadekRegularny} and Theorem~\ref{tw:przypadekKrytyczny}. The proofs are based on the idea from \cite{AptekarevGeronimo2016}. Namely, we take a sequence of approximations of $\mu$ obtained from polynomials $(p_n : n \in \NN)$ with eventually $N$-periodic sequences $a$ and $b$ (i.e. the sequences $(a_n : n \geq K)$ and $(b_n : n \geq K)$ for a number $K$ are $N$-periodic). The approximating measures can be obtained from classical theorems concerning the bounded case. Then, by detailed analysis of the proofs from \cite{GSBT} implying convergence of Turán determinants, we show that the densities of the approximations converge almost uniformly\footnote{A sequence converges almost uniformly on $X$ if it converges uniformly on every compact subset of $X$.} to the density of $\mu$.
      
      In Section~\ref{sec:convergenceChristoffel}, under additional assumptions, we show exact asymptotics of the sequence $\left(p^2_n(x) + p^2_{n+1}(x) : n \in \NN \right)$ and the convergence to $\mu'$ of appropriately scaled Christoffel functions. Let us note that the results from Section~\ref{sec:convergenceChristoffel} are mainly applicable to the case of the symmetric measure $\mu$, i.e. $b_n \equiv 0$. Additional assumptions are as follows: $r_i \equiv 1$ in the setting of Theorem~\ref{tw:przypadekRegularny} and $s_i \equiv 0$ in the setting of Theorem~\ref{tw:przypadekKrytyczny}.
      
      Finally, in Section~\ref{sec:Examples}, we check the numerical approximation by Turán determinants and we give applications to asymptotics of the sequence $\left(p^2_n(x) + p^2_{n+1}(x) : n \in \NN \right)$.
      
   \section{Preliminaries} \label{sec:Preliminaries}
      For any two sequences $(x_n : n \in \NN)$ and $(y_n : n \in \NN)$ one has
      \begin{equation}
         \label{eq:29}
         \calV_N(x_n y_n : n \in \NN) \leq \sup_{n \in \NN}{|x_n|}\ \calV_N(y_n : n \in \NN) +
         \sup_{n \in \NN}{|y_n|}\ \calV_N(x_n : n \in \NN).
      \end{equation}
      
      Observe that for every matrix $X \in M_2(\RR)$ of the form
      \[
         X = c \Id + C,
      \]
      where $c$ is a number, the following formula holds true
      \begin{equation} \label{eq:4}
         \discr(X) = \discr(C).
      \end{equation}
      For a sequence of square matrices $(C_n : n \in \NN)$ and $n_0, n_1 \in \NN$ we set
      \[
         \prod_{k=n_0}^{n_1} C_k = 
          \begin{cases} 
          C_{n_1} C_{n_1 - 1} \cdots C_{n_0} & n_1 \geq n_0, \\
          \Id & \text{otherwise.}
          \end{cases}
      \]
      
      Given sequences $a$ and $b$, we associate the sequence of polynomials $(p_n : n \in \NN)$ by 
      the formula \eqref{eq:definicjaWielomianow}. Moreover, for each $x \in \RR$ and $n \in \NN$
      we define the \emph{transfer matrix} $B_n(x)$ by
      \[
         B_n(x) = 
         \begin{pmatrix}
         0 & 1 \\
         -\frac{a_{n-1}}{a_n} & \frac{x - b_n}{a_n}
         \end{pmatrix}.
      \]      
      Then we have
      \begin{equation} \label{eq:1}
          \begin{pmatrix}
              p_n(x) \\
              p_{n+1}(x)	
         \end{pmatrix}
          =
          B_n(x)
          \begin{pmatrix}
              p_{n-1}(x) \\
              p_n(x)
          \end{pmatrix}.
      \end{equation}
      Define \emph{$N$-shifted Turán determinants} by the formula
      \[
         D^N_n(x) = p_n(x) p_{n+N-1}(x) - p_{n-1}(x) p_{n+N}(x).
      \]
      By \eqref{eq:1} we have
      \begin{equation} \label{eq:2}
         D^N_n(x) = \bigg\langle E X_n(x)          
            \begin{pmatrix}
              p_{n-1}(x) \\
              p_n(x)
            \end{pmatrix},
            \begin{pmatrix}
              p_{n-1}(x) \\
              p_n(x)
            \end{pmatrix} \bigg\rangle,
      \end{equation}
      where
      \begin{equation} \label{eq:9}
         X_n(x) = \prod_{j = n}^{n+N-1} B_j(x)
         \quad\text{and}\quad
          E = 
          \begin{pmatrix}
            0 & -1 \\
              1 & 0
          \end{pmatrix}.
      \end{equation}
      One defines Chebyshev polynomials of the second kind by the formula
      \[
         U_n(x) = \frac{\sin\big((n+1) \arccos x \big)}{\sin \arccos x}, \quad x \in (-1, 1).
      \]
      For our applications it is more convenient to use polynomials $w_n$ given by
      \begin{equation} \label{eq:defCzebyszewII}
         w_n(x) = U_n(x/2).
      \end{equation}
      
   \section{Turán determinants and the measure of orthogonality} \label{sec:TuranAndTheMeasure}
      In the case of bounded sequences $a$ and $b$ there is a close connection between Turán determinants and the density of the measure $\mu$. The first result in this direction was obtained by Nevai (see \cite[Theorem 7.34]{Nevai79} and \cite[Theorem 1]{DombrowskiNevai1986}). The following theorem is a generalization of this result (by taking $N=1$ we obtain Nevai's theorem) and was proved in \cite[Theorem~6]{JGVA}. The formulation is close to \cite[Theorem 8]{WVA}.
      \begin{tw}[Geronimo, Van Assche] \label{tw:Geronimo-VanAssche}
         Let $N$ be a positive integer. Suppose that
         \[
            \calV_N(a_n : n \in \NN) + \calV_N(b_n : n\in \NN) < \infty.
         \]
         Let
         \[
            \lim_{k \rightarrow \infty} a_{Nk+i} = \alpha_i > 0, \quad \lim_{k \rightarrow \infty} b_{Nk+i} = \beta_i
         \]
         and
         \[
            \mathcal{F}(x) = \prod_{j=0}^{N-1} 
               \begin{pmatrix}
                  0 & 1 \\
                  -\frac{\alpha_{j-1}}{\alpha_j} & \frac{x - \beta_j}{\alpha_j}
               \end{pmatrix}.
         \]
         Then for each
         \[
            x \in E = \{ y : \discr{\mathcal{F}(y)} < 0 \}
         \]
         the limit
         \[
            g(x) = \lim_{n \rightarrow \infty} a_{n+N-1} D^N_n(x)
         \]
         exists and defines a continuous function without zeros in $E$ (cf. \eqref{eq:zbieznoscTurana}).
         Moreover, the measure $\mu$ is absolutely continuous on $E$, its density can be expressed by the formula
         \[
            \mu'(x) = \frac{\sqrt{-\discr{\mathcal{F}(x)}}}{2 \pi |g(x)|}, \quad x \in E
         \]
         and the set of limit points of the support of $\mu$ is the closure of $E$.
      \end{tw}
      
      Theorem~\ref{tw:Geronimo-VanAssche} can be applied only in the case when the sequences
      $a$ and $b$ are bounded. Theorem~\ref{tw:przypadekRegularny} is a version of Theorem~\ref{tw:Geronimo-VanAssche}
      which can be used in unbounded cases. Its proof is inspired by techniques employed in \cite{AptekarevGeronimo2016}.
      
      \begin{proof}[Proof of Theorem~\ref{tw:przypadekRegularny}]
         Consider the sequence of measures $(\mu_K \colon K \in \mathbb{N})$ obtained from polynomials $(p_n^K : n \in \NN)$ associated with the "truncated" sequences $a^K$ and $b^K$, i.e. defined by
         \begin{equation} \label{eq:3}
            a^K_{n} = a_n, \quad b^K_{n} = b_n \quad (n \leq K+N-1)
         \end{equation}
         and extended $N$-periodically by
         \begin{equation} \label{eq:3cd}
            a^K_{K+kN+i} = a_{K+i}, \quad b^K_{K+kN+i} = b_{K+i} \quad (k \geq 0,\ i \in \{ 0, 1, \ldots, N-1 \}).
         \end{equation}
         Let $X^K_n$ denote the matrix defined by \eqref{eq:9} for sequences $a^K$ and $b^K$.
         
         Applying Theorem~\ref{tw:Geronimo-VanAssche} to $\mu_K$ one gets
         \begin{equation} \label{eq:wyrazenieMuKReg}
            \mu_K'(x) = \frac{\sqrt{-\discr{\mathcal{F}_K(x)}}}{2 \pi |g_K(x)|}, \quad x \in E_K,
         \end{equation}
         where by \eqref{eq:3cd} one has that $\mathcal{F}_K(x)$ and $X^K_{K+N}(x)$ are conjugated to each other. Therefore,
         \[
            \discr{\mathcal{F}_K(x)} = \discr{X^K_{N+K}(x)}.
         \]
         Moreover,
         \begin{equation} \label{eq:discrKReg}
            X^K_{K+N} =
            \prod_{j=1}^{N-1}
            \begin{pmatrix}
               0 & 1 \\
               -\frac{a_{K+j-1}}{a_{K+j}} & \frac{x - b_{K+j}}{a_{K+j}}
            \end{pmatrix} \cdot
            \begin{pmatrix}
               0 & 1 \\
               -\frac{a_{K+N-1}}{a_K} & \frac{x - b_K}{a_K}
            \end{pmatrix}.
         \end{equation}
         Let us denote
         \begin{align*}
            g_K(x) &= \lim_{n \rightarrow \infty} a^K_{n+N-1} [p^K_n(x) p^K_{n+N-1}(x) - p^K_{n-1}(x) p^K_{n+N}(x)] \\
            E_K &= \{ y : \discr{\mathcal{F}_K(y)} < 0 \}.
         \end{align*}
         
         Notice that \eqref{tw:1:1} is just Carleman's condition. Therefore, the moment problem for $\mu$ is determinate. Observe that $p_n(x) = p^K_n(x)$ for $n \leq K+N$. Hence, the moments of $\mu$ and $\mu_K$ coincide up to order $K+N$. Therefore, $\mu_K \rightarrow \mu$ in the weak-* topology. Because the measure $\mu$ is Borel, hence regular, it is sufficient to show that for every compact interval $I \subset \RR$ one has
         \begin{equation} \label{eq:zbieznoscZwartych}
            \lim_{K \rightarrow \infty} \mu_K(I) = \int_{I} \frac{\sqrt{-\discr{\mathcal{F}}}}{2 \pi |g(x)|} \ud x.
         \end{equation}
         Fix a compact interval $I \subset \RR$. From now, the term uniform convergence refers to the uniform convergence on $I$.
         
         Observe that \eqref{eq:7} implies that $(a_{n+N}/a_n : n \in \NN)$ converges to a number $a$. Condition \eqref{tw:1:2} forces that $a \geq 1$, whereas \eqref{tw:1:1} show that $a \leq 1$. Hence, 
         \begin{equation} \label{eq:granicaIlorazuAReg}
            \lim_{n \rightarrow \infty} \frac{a_{n+N}}{a_n} = 1.
         \end{equation}
         Since
         \begin{equation} \label{eq:ilorazSklejenia}
            \frac{a_{K+N-1}}{a_K} = \frac{a_{K-1}}{a_{K}} \frac{a_{K+N-1}}{a_{K-1}}
         \end{equation}
         and \eqref{tw:1:2}, by \eqref{eq:discrKReg} and \eqref{eq:granicaIlorazuAReg} one gets $\discr{\mathcal{F}_K(x)} \rightarrow \discr{\mathcal{F}}$ almost uniformly. Therefore, the right-hand side of the numerator of \eqref{eq:wyrazenieMuKReg} converges to $\sqrt{-\discr{\mathcal{F}}}$. Consequently, \eqref{eq:regWarNaDiscr} implies that there exists $K_0$ such that for all $K \geq K_0$ one has $I \subset E_K$. Then, by Theorem~\ref{tw:Geronimo-VanAssche} the measure $\mu_K$ is absolutely continuous on $I$. Hence,
         \[
            \mu_K(I) = \int_I \mu_K'(x) \ud x.
         \]
         Therefore, the condition \eqref{eq:zbieznoscZwartych} holds as long as
         \begin{equation} \label{eq:zbieznoscGK}
            \lim_{K \rightarrow \infty} g_K(x) = g(x)
         \end{equation}
         uniformly on $I$ and $|g(x)|$ is a continuous and positive function.
         
         Let us define
         \[
            S_n(x) = a_{n+N-1} D^N_n(x), \quad S^K_n(x) = a^K_{n+N-1} [p^K_n(x) p^K_{n+N-1}(x) - p^K_{n-1}(x) p^K_{n+N}(x)].
         \]
         Formulas (21), (22) from \cite{GSBT} imply
         \begin{multline} \label{eq:oszacowanieFn}
            | S^K_{n+1}(x) - S^K_n(x) | \leq \lVert X^K_n(x) \rVert \bigg( a^K_{n-1} \bigg|\frac{a^K_{n+N}}{a^K_n} - \frac{a^K_{n+N-1}}{a^K_{n-1}} \bigg| + |x| a^K_{n+N} \bigg| \frac{1}{a^K_{n+N}} - \frac{1}{a^K_n} \bigg| \\
            + a^K_{n+N} \bigg| \frac{b^K_{n+N}}{a^K_{n+N}} - \frac{b^K_{n}}{a^K_{n}} \bigg| \bigg) \left( (p^K_{n-1}(x))^2 + (p^K_{n}(x))^2 \right).
         \end{multline}
         Therefore, by \eqref{eq:3cd}, $S^K_{n+1} = S^K_n$ for $n \geq K+1$, and consequently, $g_K(x) = S^K_{K+1}(x)$.
         
         Observe, that equations \eqref{eq:2} and \eqref{eq:3} show
         \[
            X^K_K(x) = X_K(x), \quad S^K_K(x) = S_K(x).
         \]
         Then \eqref{eq:oszacowanieFn} gives
         \[
            | S^K_{K+1}(x) - S^K_K(x) | \leq \lVert X_K(x) \rVert a_{K-1} \bigg| 1 - \frac{a_{K+N-1}}{a_{K-1}} \bigg| ( p^2_{K-1}(x) + p^2_{K}(x) ).
         \]
         Formula (23) from \cite{GSBT} assures that
         \[
            |S_K(x)| \geq c a_{N+K-1} \left( p^2_{K-1}(x) + p^2_{K}(x) \right)
         \]
         for a constant $c>0$. Therefore, by \eqref{eq:granicaIlorazuAReg}
         \[
            F^K_K(x) := \frac{S^K_{K+1}(x) - S^K_K(x)}{S_K(x)}
         \]
         tends to $0$ uniformly. By formula (24) from \cite{GSBT}
         \[
            F_K(x) := \frac{S_{K+1}(x) - S_K(x)}{S_K(x)}
         \]
         also tends to $0$ uniformly (it is even summable). Therefore,
         \[
            \frac{S^K_{K+1}(x)}{S_{K+1}(x)} = \frac{1 + F^K_K(x)}{1 + F_K(x)}
         \]
         tends to $1$ uniformly. Consequently, 
         \[
            \lim_{K \rightarrow \infty} g_K(x) = \lim_{K \rightarrow \infty} S_{K+1}(x)
         \]
         uniformly. Finally, \cite[Corollary 2]{GSBT} implies that $S_n(x)$ converges uniformly to $g(x)$, which is a continuous function without zeros. It shows \eqref{eq:zbieznoscGK}. The proof is complete.
      \end{proof}
      
      We are ready to prove Theorem~\ref{tw:przypadekKrytyczny}, in which one has $\discr{\mathcal{F}} = 0$, so the previous result cannot be applied.

      \begin{proof}[Proof of Theorem~\ref{tw:przypadekKrytyczny}]
         We proceed similarly to the method used in the proof of Theorem~\ref{tw:przypadekRegularny}. We consider the sequence of measures $(\mu_K \colon K \in \mathbb{N})$ obtained from the polynomials $(p_n^K : n \in \NN)$ associated with the truncated sequences $a^K$ and $b^K$ (see \eqref{eq:3}).
         By the application of Theorem~\ref{tw:Geronimo-VanAssche} to $\mu_K$ we obtain
         \[
            \mu_K'(x) = \frac{\sqrt{-\discr(\mathcal{F}_K(x))}}{2 \pi |g_K(x)|}, \quad x \in E_K,
         \]
         where $\discr(\mathcal{F}_K(x)) = \discr{X^K_{N+K}(x)},\ g_K(x)$ and $E_K$ are the same as before.
         Again, it is sufficient to show that for every compact interval $I \subset \RR \backslash [x_-,x_+]$
         \begin{equation} \label{eq:5}
            \lim_{K \rightarrow \infty} \mu'_K(x) = \frac{\sqrt{h(x)}}{2 \pi |\widetilde{g}(x)|}
         \end{equation}
         uniformly on $I$ and $|\widetilde{g}(x)|$ is a continuous positive function. In fact, because the sequence $(\mu_K : K \in \NN)$ is convergent, it is sufficient to show the convergence on one subsequence.
         
         First, we want to compute the value of the limit of $\discr(\mathcal{F}_K(x))$. Notice that formula (43) from \cite{GSBT} shows decomposition
         \begin{equation} \label{eq:krytXN}
            X^K_{K+N}(x) = (-1)^{k_0 + N} \Id  + \frac{1}{a^K_{K+2N-1}} C^K_{K+N}(x).
         \end{equation}
         Our aim is to show that one has almost uniform convergence
         \begin{equation} \label{eq:rownowaznoscDiscrKryt}
            \lim_{K \rightarrow \infty} (C^K_{K+N}(x) - C_K(x)) = 0,
         \end{equation}
         where $C_K(x)$ is defined by
         \[
            C_K(x) = a_{K+N-1}(X_K - (-1)^{N+k_0} \Id).
         \]
         
         Let $B^K_n(x)$ denotes the transfer matrix associated with sequences $a^K$ and $b^K$. 
         The formula (43) from \cite{GSBT} expressing $C^K_{K+N}(x)$ involves the following matrices:
         \begin{itemize}
            \item $B^K_{K+N+j}(x)$ for $j \in \{ 0, 1, \ldots, N-2 \}$,
            
            \item $a^K_{K+2N-1}(B^K_{K+N+j}(x) - \tilde{B}^K_{K+N}(x))$  for $j \in \{ 0, 1, \ldots, N-1 \}$ and
            \[
               \tilde{B}^K_{K+N}(x) = 
               \begin{pmatrix}
                  0 & 1 \\
                  -1 & \frac{x}{a^K_{K+N}} - q
               \end{pmatrix}
               = 
               \begin{pmatrix}
                  0 & 1 \\
                  -1 & \frac{x}{a_{K}} - q
               \end{pmatrix} = \tilde{B}_{K}(x),
            \]
            
            \item $\tilde{C}^K_{K+N}(x) = a^K_{K+N} ((\tilde{B}^K_{K+N}(x))^N - (-1)^{k_0+N} \Id)$, which equals $\tilde{C}_K(x)$.
         \end{itemize}
         One has that $B^K_{K+N+j}(x) = B_{K+j}(x)$ for $0<j<N$ and
         \[
            B^K_{K+N}(x) = 
            \begin{pmatrix}
               0 & 1 \\
               -\frac{a_{K+N-1}}{a_K} & \frac{x - b_K}{a_K}
            \end{pmatrix}.
         \]
         Since the sequence $(a_n - a_{n-1} : n \in \NN)$ is bounded 
         \begin{equation} \label{eq:granicaIlorazuAKryt}
            \lim_{n \rightarrow \infty} \frac{a_{n+1}}{a_n} = 1,
         \end{equation}
         which, by \eqref{eq:ilorazSklejenia}, shows that 
         \[
            \lim_{K \rightarrow \infty} (B^K_{K+N}(x) - B_K(x)) = 0
         \]
         almost uniformly.
         Similarly, for $0<j<N$ one has
         \[
            a^K_{K+2N-1}(B^K_{K+N+j}(x) - \tilde{B}^K_{K+N}(x)) = a_{K+N-1} (B_{K+j}(x) - \tilde{B}_K(x))
         \]
         and by (48) from \cite{GSBT}
         \[
            a^K_{K+2N-1}(B^K_{K+N}(x) - \tilde{B}^K_{K+N}(x)) = 
            \frac{a_{K+N-1}}{a_K}
            \begin{pmatrix}
               0 & 0 \\
               a_K - a_{K+N-1} & q a_K - b_K
            \end{pmatrix}.
         \]
         Observe,
         \begin{equation} \label{eq:10}
            a_K - a_{K+N-1} = (a_K - a_{K-1}) - (a_{K+N-1} - a_{K-1}),
         \end{equation} 
         which combined with \eqref{eq:granicaIlorazuAKryt} and assumption \eqref{tw:2:2} shows that
         \[
            \lim_{K \rightarrow \infty} (a^K_{K+2N-1}(B^K_{K+N}(x) - \tilde{B}^K_{K+N}(x)) - a_{K+N-1}(B_{K}(x) - \tilde{B}_{K}(x))) = 0
         \]
         almost uniformly. This completes the proof of \eqref{eq:rownowaznoscDiscrKryt}.
         
         Let
         \[
            S_n(x) = a_{n+N-1}^2 D^N_n(x), \quad S^K_n(x) = (a^K_{n+N-1})^2 [p^K_n(x) p^K_{n+N-1}(x) - p^K_{n-1}(x) p^K_{n+N}(x)].
         \]
         By \eqref{eq:3cd} and \eqref{eq:krytXN}, matrices $\mathcal{F}_K(x)$ and $C^K_{N+K}$ are conjugated to each other. Hence, by \eqref{eq:4} and \eqref{eq:krytXN},
         \[
            \discr{\mathcal{F}_K(x)} = \frac{1}{a^2_{K+N-1}} \discr{C^K_{N+K}(x)}.
         \]
         By \eqref{eq:rownowaznoscDiscrKryt}, for every $i$
         \begin{equation} \label{eq:granicaGestosciKryt}
            \lim_{k \rightarrow \infty} \frac{\sqrt{-\discr{\mathcal{F}_{Nk+i}(x)}}}{a_{N(k+1)+i-1} D^N_{Nk+i}(x)} = \lim_{k \rightarrow \infty} \frac{\sqrt{-\discr{C_{Nk+i}(x)}}}{S_{Nk+i}(x)} = \frac{\sqrt{h(x)}}{\widetilde{g}^{i}(x)},
         \end{equation}
         where the last equality follows from equation (53) from \cite{GSBT} and \cite[Corollary 3]{GSBT}.
         
         Let
         \[
            \widetilde{g}^{i}_K(x) = \lim_{k \rightarrow \infty} S^K_{Nk + i}(x).
         \]
         We consider $K = kN+i$ for $0 \leq i < N$. Hence, equation \eqref{eq:5} will be satisfied if we show
         \begin{equation} \label{eq:zbieznoscGKrytyczne}
            \lim_{k \rightarrow \infty} \widetilde{g}^{i}_{Nk+i}(x) = \widetilde{g}^{i}(x)
         \end{equation}
         uniformly on $I$. In particular, \eqref{eq:granicaGestosciKryt} will imply that the limit defining $\widetilde{g}(x)$ exists and $\widetilde{g}^i(x) = \widetilde{g}(x)$ for every $0 \leq i < N$.  
         
         Formulas (54) and (55) from \cite{GSBT} give
         \begin{equation} \label{eq:6}
            |S^K_{n+N}(x) - S^K_n(x)| \leq \bigg\lVert a^K_{n+2N-1} C^K_{n+N}(x) - \frac{a^2_{n+N-1}}{a_{n-1}} C^K_n(x) \bigg\rVert ((p^K_{n+N-1}(x))^2 + (p^K_{n+N}(x))^2).
         \end{equation}
         Therefore, by \eqref{eq:3cd}, $S^K_{n+N} = S^K_n$ for $n \geq K+1$. Take $n=K+N$, then $\widetilde{g}^i_K(x) = S^K_{K+N}(x)$. Observe, that formulas \eqref{eq:2} and \eqref{eq:3} force that
         \[
            X^K_{K}(x) = X_{K}(x), \quad C^K_K(x) = C_K(x), \quad S^K_{K}(x) = S_{K}(x).
         \]
         Then \eqref{eq:6} implies
         \[
            |S^K_{K+N}(x) - S^K_K(x)| \leq a_{K+N-1} \left| 1 - \frac{a_{K+N-1}}{a_{K-1}} \right| \lVert C^K_K(x) \rVert (p^2_{K+N-1}(x) + p^2_{K+N}(x)).
         \]
         Formula (56) from \cite{GSBT} shows
         \[
            |S_K(x)| \geq c \frac{a^2_{K+N-1}}{a_{K-1}} (p^2_{K+N-1}(x) + p^2_{K+N}(x))
         \]
         for a constant $c>0$. Therefore, by \eqref{eq:granicaIlorazuAReg}
         \[
            F^K_K(x) := \frac{S^K_{K+N}(x) - S^K_K(x)}{S_K(x)}
         \]
         tends to $0$ uniformly as long as $\lVert C^K_K(x) \rVert$ is uniformly bounded. Indeed, formulas \eqref{eq:9} and \eqref{eq:krytXN} imply
         \[
            C^K_{K+N}(x) = C^K_K(x) + \frac{a_{K+N-1}}{a_K} \prod_{j=K+1}^{K+N-1} B_j(x) \cdot
            \begin{pmatrix}
               0 & 0 \\
               a_{K-1} - a_{K+N-1} & 0
            \end{pmatrix},
         \]
         which by \eqref{eq:rownowaznoscDiscrKryt}-\eqref{eq:10} and assumption \eqref{tw:2:1} implies the uniform boundedness of $\lVert C^K_K(x) \rVert$.
         By formula (57) from \cite{GSBT}
         \[
            F_K(x) := \frac{S_{K+N}(x) - S_K(x)}{S_K(x)}
         \]
         also tends to $0$ uniformly (it is even summable). Therefore,
         \[
            \frac{S^K_{K+N}(x)}{S_{K+N}(x)} = \frac{1 + F^K_K(x)}{1 + F_K(x)}
         \]
         tends to $1$ uniformly. Consequently, 
         \[
            \lim_{k \rightarrow \infty} \tilde{g}^i_{kN+N+i}(x) = \lim_{k \rightarrow \infty} S_{kN+N+i}(x)
         \]
         uniformly. Finally, \cite[Corollary 3]{GSBT} implies that $(S_{kN+i}(x) : k \in \NN)$ converges uniformly to $\tilde{g}^i(x)$, which is a continuous function without zeros. It shows \eqref{eq:zbieznoscGKrytyczne}. The proof is complete.
      \end{proof}
   
   \section{Convergence of Christoffel functions} \label{sec:convergenceChristoffel}
      Let $\lambda_n$ be the $n$th \emph{Christoffel function}, i.e.
      \[
         \lambda_n(x) = \bigg[ \sum_{k=0}^{n} p_k^2(x) \bigg]^{-1}.
      \]
      Christoffel functions have applications in approximation theory and in random matrix theory (see, e.g. \cite{Nevai19863}, \cite{Lubinsky}).
   
      We begin with corollaries which provide exact asymptotics of orthonormal polynomials under stronger assumptions 
      than Theorem~\ref{tw:przypadekRegularny} and Theorem~\ref{tw:przypadekKrytyczny}.
      
      \begin{wn}[Regular case] \label{tw:asymptReg}
         Let $N$ be an odd number. Assume that
         \[
             \calV_1\bigg(\frac{a_{n+N}}{a_n} : n\in \NN\bigg) + \calV_N\bigg(\frac{1}{a_n} : n\in \NN\bigg) + \calV_N\bigg(\frac{b_n}{a_n} : n\in \NN\bigg) < \infty.
         \]
         If
         \begin{enumerate}[(a)]
            \item
            $\begin{aligned}
               \lim_{n \rightarrow \infty} \frac{a_{n+1}}{a_n} = 1, \quad \lim_{n \rightarrow \infty} a_n = \infty, \quad \lim_{n \rightarrow \infty} \frac{b_n}{a_n} = 0,
            \end{aligned}$
            
            \item
            $\begin{aligned}
               \sum_{n=0}^\infty \frac{1}{a_n} = \infty,
            \end{aligned}$            
         \end{enumerate}
         then
         \begin{equation} \label{eq:asymptReg}
            \lim_{n \rightarrow \infty} a_n [p_{n-1}^2(x) + p_n^2(x)] = [\pi \mu'(x)]^{-1}, \quad x \in \RR
         \end{equation}
         almost uniformly.
      \end{wn}
      \begin{proof}
         Observe that $B_n(x)$ converges to $-E$ almost uniformly. Hence, $\det{\mathcal{F}} = 1, \tr{\mathcal{F}} = 0$ and consequently, $\discr{\mathcal{F}} = -4$. Moreover, $E X_n$ tends to $(-1)^{\lfloor N/2 \rfloor} \Id$. Therefore, by \cite[Corollary 1]{GSBT}
         \[
            \lim_{n \rightarrow \infty} a_n [p_{n-1}^2(x) + p_n^2(x)] = |g(x)|.
         \]
         Theorem~\ref{tw:przypadekRegularny} implies
         \[
            |g(x)| = [\pi \mu'(x)]^{-1},
         \]
         which is what had to be proven.
      \end{proof}
      
      \begin{wn}[Critical case] \label{tw:asymptKryt}
         Let $N$ be an even number. Assume that
         \[
            \calV_N\bigg(a_{n} - a_{n-1} : n\in \NN\bigg) + \calV_N\bigg(\frac{1}{a_n} : n\in \NN\bigg) + \calV_N\bigg(b_n : n\in \NN\bigg) < \infty.
         \]
         If
         \[
            \lim_{n \rightarrow \infty} (a_n - a_{n-1}) = 0, \quad \lim_{n \rightarrow \infty} a_n = \infty, \quad \lim_{n \rightarrow \infty} b_n = 0,
         \]
         then
         \begin{equation} \label{eq:asymptKryt}
            \lim_{n \rightarrow \infty} a_n [p_{n-1}^2(x) + p_n^2(x)] = [\pi \mu'(x)]^{-1}, \quad x \in \mathbb{R} \backslash \{ 0 \}
         \end{equation}
         almost uniformly.
      \end{wn}
      \begin{proof}
         Observe that formulas (35) and (49) from \cite{GSBT} imply that $C_n(x)$ converges to $\frac{x N}{2} (-1)^{N/2} E$ almost uniformly. Therefore, $E C_n(x)$ tends to $\frac{x N}{2} (-1)^{N/2 + 1} \Id$. Hence, by \cite[Corollary 1]{GSBT}
         \[
            \lim_{n \rightarrow \infty} a_n [p_{n-1}^2(x) + p_n^2(x)] = \frac{2 |\widetilde{g}(x)|}{|x| N}.
         \]
         Theorem~\ref{tw:przypadekRegularny} implies
         \[
            |\widetilde{g}(x)| = \frac{|x N|}{2 \pi \mu'(x)}. \qedhere
         \]
      \end{proof}
      
      In \cite{simon2007orthogonal} it was proven that \emph{any} measure $\mu$ for which the moment problem is determinate satisfies
      \[
         \lim_{n \rightarrow \infty} \frac{\ud x}{a_n^2 p_n^2(x) + p_{n-1}^2(x)} = \pi \ud \mu(x)
      \]
      in the weak-* topology (Carmona-Simon formula). It somehow resembles \eqref{eq:asymptReg} and \eqref{eq:asymptKryt}. But in this case the denominator of the left-hand-side is different: $a_n^2 p_n^2(x) + p_{n-1}^2(x)$ instead of $a_n p_{n-1}^2(x) + a_n p_n^2(x)$. Moreover, Carmona-Simon formula can be applied even for singular measures, hence it gives no information about the density of $\mu$. 
      
      We now turn to applications of the previous corollaries to the convergence of Christoffel functions. 
      
      \begin{wn} \label{wn:zbieznoscChristoffel}
         Under the conditions of Corollary~\ref{tw:asymptReg} or Corollary~\ref{tw:asymptKryt} one has
         \begin{equation} \label{eq:zbieznoscChristoffel}
            \lim_{n \rightarrow \infty} \lambda_n(x) \sum_{k=0}^{n} \frac{1}{a_k} = 2 \pi \mu'(x)
         \end{equation}
         almost uniformly on $\RR$ in the regular case and on $\RR \backslash \{ 0 \}$ in the critical case.
      \end{wn}
      \begin{proof}
         By Stolz theorem (see \cite[pages 173–175]{Stolz1885})
         \[
            \lim_{n \rightarrow \infty} \lambda_n(x) \sum_{k=0}^{n} \frac{1}{a_k} = \lim_{n \rightarrow \infty} \frac{1/a_{n-1} + 1/a_n}{p^2_{n-1}(x) + p^2_n(x)} = \lim_{n \rightarrow \infty} \frac{a_n/a_{n-1} + 1}{a_n[p^2_{n-1}(x) + p^2_n(x)]},
         \]
         which by Corollary~\ref{tw:asymptReg} or Corollary~\ref{tw:asymptKryt} equals $2 \pi \mu'(x)$.
      \end{proof}
      
      It turns out that the right-hand side of \eqref{eq:zbieznoscChristoffel} is exactly the same as in the case of the unit circle (see \cite[Theorem 1]{MNT91}). In \cite{MNT91} conditions assuring the convergence were imposed on the measure $\mu$ not on the recurrence relation. The convergence of Christoffel functions to a wide class of measures with compact support on the real line was obtained in \cite{VT2009}.
      
      In \cite[Conjecture 6.7]{IgnA} the following conjecture has been stated.
      \begin{hipoteza}[Ignjatović \cite{IgnA}]
         \label{con:1}
         Assume that $b_n \equiv 0$ and
         \begin{equation}
            \label{eq:45}
            \lim_{n \rightarrow \infty} \frac{a_n}{n^\kappa} = c > 0, \quad \kappa \in (0, 1).
         \end{equation}
         Then 
         \begin{equation}
            \label{eq:44}
            \lim_{n \rightarrow \infty} \frac{\sum_{k=0}^n p^2_k(x)}{\sum_{k=0}^n 1/a_k}
         \end{equation}
         exists and is positive for $x \in \supp(\mu)$.
      \end{hipoteza}
      Let us observe that some additional assumptions to this conjecture are needed. Indeed, consider
      \[
         a_0 = \epsilon > 0, \quad a_{2k-1} = a_{2k} = k^{\kappa}, \quad k \geq 1, \quad \kappa \in (1/2, 1).
      \]
      Then, by \cite{MM1}, the sequence $(p_n(0) : n \in \NN)$ is square summable and $0 \in \supp(\mu) = \RR$. 
      Therefore, the limit \eqref{eq:44} equals zero. For a more general example, see \cite[Example 6.1]{GS1}.
      On the other hand it can be infinite. To see this, it is enough to consider 
      \[
         a_{2k} = a_{2k+1} = (k+1)^{\kappa}, \quad k \geq 0, \quad \kappa \in (0, 1),
      \]
      then $p_{2k}(0) = (-1)^k$ and $p_{2k+1}(0) = 0$, and hence, by Stolz theorem
      \[
         \lim_{n \rightarrow \infty} \frac{\sum_{k=0}^n p^2_k(0)}{\sum_{k=0}^n 1/a_k} = \lim_{n \rightarrow \infty} \frac{p^2_{n-1}(0) + p^2_n(0)}{1/a_{n-1} + 1/a_n} = \lim_{n \rightarrow \infty} \frac{1}{1/a_{n-1} + 1/a_n} = \infty.
      \]
      But according to \cite{MM1} one has $0 \in \supp(\mu) = \mathbb{R}$. A more general example may be found in \cite[Example 4]{GSBT}.
      
      Observe that Corollary~\ref{wn:zbieznoscChristoffel} gives not only sufficient conditions when the conjecture is correct
      but also provides the value of the limit. This value is important in the applications (see \cite[Proposition 6.6]{IgnA}).
      For a result in this direction see \cite{IgnA2}.
      
      Recently in \cite{IgnA1} Conjecture \ref{con:1} was resolved under some additional conditions. However, they are
      stronger than the hypothesis of Corollary~\ref{wn:zbieznoscChristoffel} for $N = 1$. Indeed, we have 
      \[
         \frac{a_{n+1} - a_n}{a_n^2}
         =
         \bigg(\frac{1}{a_n} - \frac{1}{a_{n+1}}\bigg) \frac{a_{n+1}}{a_n},
      \]
      thus
      \begin{equation}
         \tag{$C_6$}
         \sum_{n = 0}^\infty \frac{|a_{n+1} - a_n|}{a_n^2} < \infty,
      \end{equation}
      together with \eqref{eq:45}, implies that
      \[
         \calV_1\bigg(\frac{1}{a_n} : n \in \NN \bigg) < \infty.
      \]
      Since
      \[
         \frac{a_{n+2} - 2 a_{n+1} + a_n}{a_n} = (a_{n+2} - a_{n+1}) \bigg(\frac{1}{a_n} - \frac{1}{a_{n+1}}\bigg)
         + \bigg(\frac{a_{n+2}}{a_{n+1}} - \frac{a_{n+1}}{a_n}\bigg),
      \]
      the condition 
      \begin{equation}
         \tag{$C_2$}
         \lim_{n \to \infty} (a_{n+1} - a_n) = 0
      \end{equation}
      together with
      \begin{equation}
         \tag{$C_7$}
         \sum_{n = 0}^\infty \frac{|a_{n+2} - 2 a_{n+1} + a_n|}{a_n} < \infty
      \end{equation}
      gives 
      \[
         \calV_1\bigg(\frac{a_{n+1}}{a_n} : n \in \NN \bigg) < \infty.
      \]
      From the other side, conditions $(C_2)$ and $(C_3)$ are not necessary. For example, taking $a_n = n+1$, we have
      \[
         \calV_1\bigg(\frac{a_{n+1}}{a_n} : n \in \NN\bigg) < \infty,
      \]
      thus we can apply Corollary~\ref{wn:zbieznoscChristoffel}, but $a_{n+1} - a_n = 1$ and the condition $(C_2)$ is not satisfied. 

      We want also to stress that the sequence $a_n = \log (n+2)$ does not satisfy \eqref{eq:45}, but Corollary~\ref{wn:zbieznoscChristoffel} is
      applicable since
      \[
         \calV_1(a_{n+1} - a_n : n \in \NN) + \calV_1\bigg(\frac{1}{a_n} : n \in \NN\bigg) < \infty.
      \]
      
   \section{Examples} \label{sec:Examples}
      This section is devoted to illustrate the numerical usefulness of the formulas from Theorem~\ref{tw:przypadekRegularny} and Theorem~\ref{tw:przypadekKrytyczny}. The computations were performed on Maple 18\footnote{The file with the program is available on the author's website: \url{http://www.math.uni.wroc.pl/~gswider/research/}}. Furthermore, we show that results from Section~\ref{sec:convergenceChristoffel} sometimes can be used even when one does not have the exact values of the sequences $a$ and $b$.
      
      \begin{przyklad} \label{ex:1}
         Let $p_n$ be a sequence of \emph{generalized Hermite polynomials} (see \cite[p. 157]{Chihara5}), i.e.
         \[
            a_n = \frac{1}{\sqrt{2}} \sqrt{n+1+d_n}, \quad d_{2k} = t, \quad d_{2k+1} = 0, \quad b_n \equiv 0, 
         \]
         where $t > -1$. Then 
         \[
            \mu'(x) = \frac{1}{\Gamma((1+t)/2)} |x|^{t} \ue^{-x^2}.
         \]
         
         The assumptions of Corollary~\ref{tw:asymptKryt} for $N=2$ are satisfied. Indeed, from Taylor's formula one gets
         \[
            a_n - a_{n-1} = \frac{1}{2\sqrt{2}} (1 + d_n - d_{n-1}) \frac{1}{\sqrt{n+d_{n-1}}} + r_n, \quad (r_n) \in \ell^1.
         \]
         This, combined with 2-periodicity of $d_n$ and Taylor's formula, implies that the sequence $(a_n-a_{n-1} : n \geq 1)$ has bounded total $2$-variation. Finally, the sequence $(1/a_n : n \in \NN)$ has also bounded total $2$-variation because $a_{n+2} > a_n$ for every $n$. Therefore
         \[
            \lim_{n \rightarrow \infty} a_n [p_{n-1}^2(x) + p_n^2(x)] = \frac{\Gamma((1+t)/2)}{\pi} |x|^{-t} \ue^{x^2}
         \]
         almost uniformly on $\RR \backslash \{ 0 \}$.
         
         For $t < 0$ the density is not continuous at $0$. For $t \geq 0$ the density is $C^{\lceil t \rceil - 1}(\mathbb{R})$, with the exception when $t \in 2 \NN$. In this case the density is $C^\infty(\RR)$. It shows that the restriction of Theorem~\ref{tw:przypadekKrytyczny} to the interval $\RR \backslash [x_-,x_+]$ instead of $\RR \backslash (x_-,x_+)$ is necessary.
         
         Observe that assumptions of Theorem~\ref{tw:przypadekRegularny} cannot be satisfied for $t \neq 0$, because of the behaviour of $\mu'$ at $0$.
         
         Table~\ref{tab:1} contains the results of the approximation of the density by the formula from Theorem~\ref{tw:przypadekKrytyczny}, whereas on the Figure~\ref{img:1}, there is a graph of the relative error of the approximation. Let us observe that for $t \neq 0$ the neighbourhood of the point $0$ is particularly difficult to mimic by the approximation. It is caused by the singularity of $\mu'$ at $0$. Overall, away from $0$ errors are reasonably small.
         
         \begin{table}
            \label{tab:1}
            \begin{center}
               \begin{tabularx}{\textwidth}{|l||X|X|X|X|X|X|} \hline
                  $\pmb{t \backslash n}$ & \textbf{10} & \textbf{20} & \textbf{40} & \textbf{60} & \textbf{80} & \textbf{100} \\ \hline \hline
                  \textbf{-0.5} & 1.32e-1 & 7.17e-2 & 6.68e-2 & 7.36e-2 & 6.84e-2 & 5.78e-2 \\
                  \textbf{0.0 (N=1)} & 1.68e-1 & 7.21e-2 & 3.55e-2 & 2.17e-2 & 1.70e-2 & 1.39e-2 \\
                  \textbf{0.0} & 1.18e-1 & 5.59e-2 & 2.97e-2 & 2.02e-2 & 1.53e-2 & 1.23e-2 \\
                  \textbf{0.5} & 4.12e-1 & 2.66e-1 & 1.30e-1 & 6.61e-2 & 2.98e-2 & 2.42e-2 \\
                  \textbf{1.0} & 1.57e0 & 9.84e-1 & 5.25e-1 & 3.21e-1 & 2.06e-1 & 1.33e-1 \\ \hline
               \end{tabularx}
            \end{center}
            \caption{The maximal relative errors of the approximation of generalized Hermite weight via $D_n^N$. If not specified, then $N=2$.}
         \end{table}
         
         \begin{figure}
           \label{img:1}
           \caption{The left-hand side: the graph of generalized Hermite weight for $t=1/2$. The right-hand side: the graph of the relative error of the approximation by $D^2_{80}$.}
           \centering
             \includegraphics[width=0.49\textwidth]{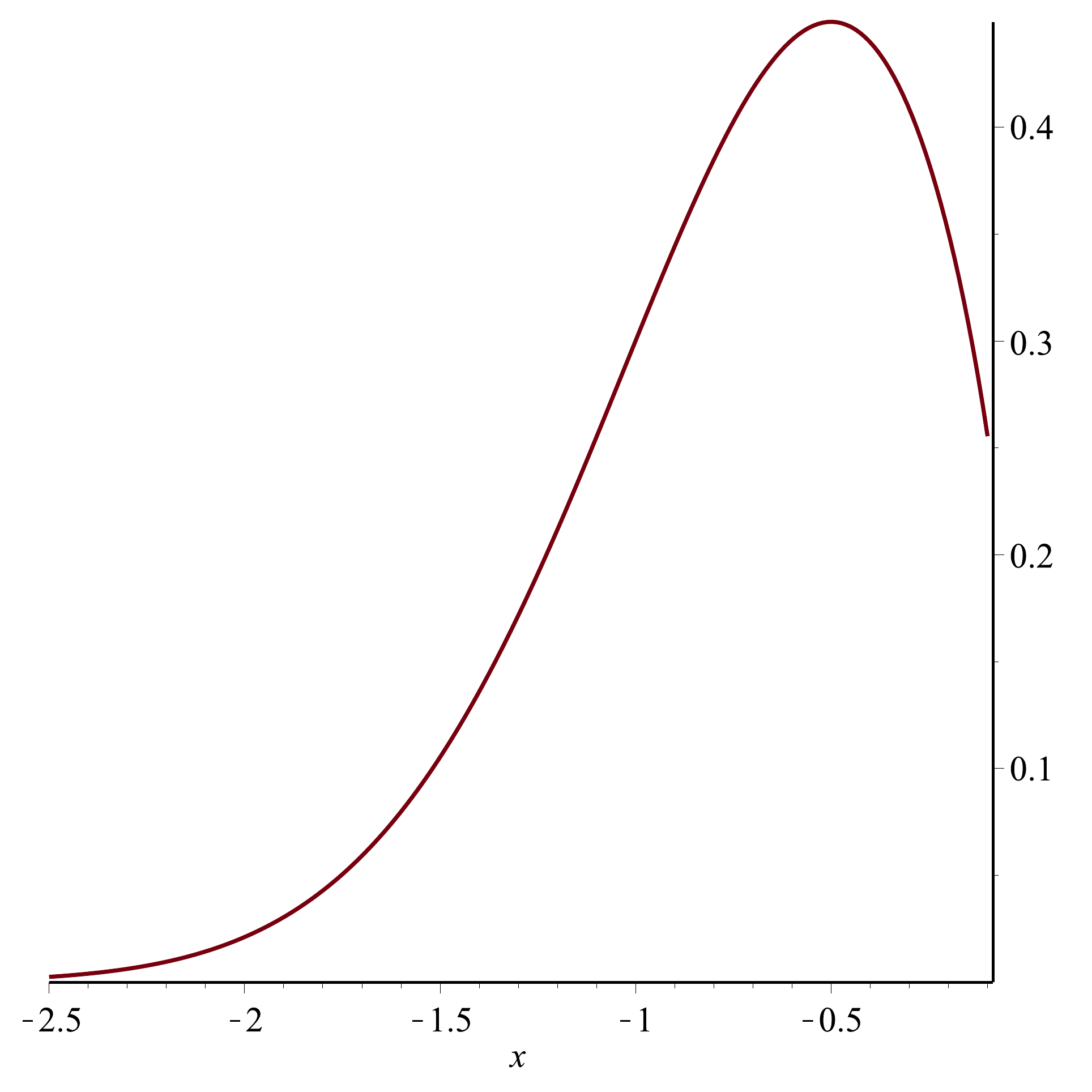}
             \includegraphics[width=0.49\textwidth]{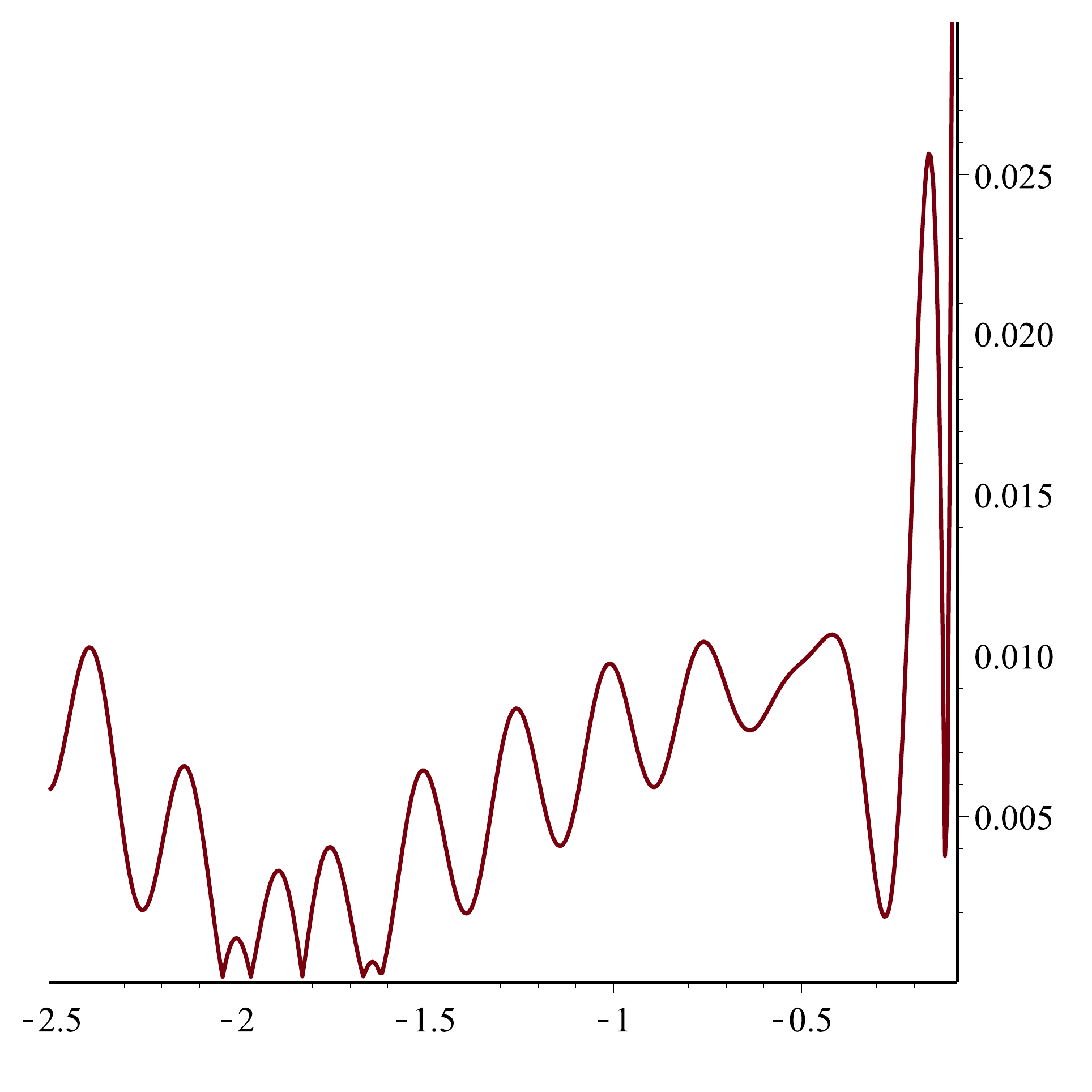}
         \end{figure}
      \end{przyklad}
         
      \begin{przyklad} \label{ex:2}
         Let $p_n$ be a sequence of \emph{Meixner-Pollaczek polynomials} (see \cite[Chapter 9.7]{Askey98}), i.e.
         \[
            a_n = \frac{\sqrt{(n+1) (n + 2 \lambda)}}{2 \sin \varphi}, \quad b_n = \frac{n + \lambda}{\tan \varphi} \quad (\lambda > 0,\ \varphi \in (0, \pi)).
         \]
         Then
         \[
            \mu'(x) = \frac{(2 \sin \varphi)^{2 \lambda}}{2 \pi \Gamma(2 \lambda)} \ue^{(2 \varphi - \pi) x} |\Gamma(\lambda + i x)|^2.
         \]
         
         The assumptions of Theorem~\ref{tw:przypadekRegularny} for $N=1$ are satisfied. Indeed, since the sequence $(a_n : n \in \NN)$ is increasing, the sequence $(1/a_n : n \in \NN)$ has bounded total $1$-variation. Because
         \[
            \frac{a_{n+1}}{a_{n}} = \sqrt{\frac{n+2}{n+1}} \sqrt{\frac{n+1 + 2\lambda}{n + 2\lambda}}
         \]
         and for every $a > 0$ the sequence $\widetilde{a}_n = \sqrt{n+a}$  satisfies the conditions of Theorem~\ref{tw:przypadekRegularny} with $N=1$, by \eqref{eq:29}, we obtain that the sequence $(a_{n+1}/a_n : n \in \NN)$ has bounded total $1$-variation. Finally, one has
         \[
            \frac{b_n}{a_n} = 2 \cos{\varphi} \frac{n+\lambda}{\sqrt{(n + 1)(n + 2\lambda)}}
         \]
         and because it has constant sign, it is enough to show that its square is eventually increasing. Indeed, 
         \[
            \frac{(n+\lambda)^2}{(n + 1)(n + 2\lambda)} = 1 - \frac{n + 2 \lambda - \lambda^2}{(n + 1)(n + 2\lambda)}.
         \]
         
         In Table~\ref{tab:2} there are results of the approximation of the density by the formula from Theorem~\ref{tw:przypadekRegularny}, whereas on the Figure~\ref{img:2}, there is a graph of the relative error of the approximation, which looks particularly simple. In all of the cases, relative errors are reasonably small from $n=40$.
         
         \begin{table}
            \label{tab:2}
            \begin{center}
               \begin{tabularx}{\textwidth}{|l||X|X|X|X|X|X|} \hline
                  $\pmb{\lambda, \varphi \backslash n}$ & \textbf{10} & \textbf{20} & \textbf{40} & \textbf{60} & \textbf{80} & \textbf{100} \\ \hline \hline
                  \textbf{0.5, $\pmb{\pi}$/4} & 1.27e-1 & 6.53e-2 & 3.56e-2 & 2.43e-2 & 1.84e-2 & 1.48e-2 \\
                  \textbf{0.5, $\pmb{\pi}$/3} & 9.29e-2 & 4.84e-2 & 2.58e-2 & 1.76e-2 & 1.32e-2 & 1.05e-2 \\
                  \textbf{0.5, $\pmb{\pi}$/2} & 4.86e-2 & 2.47e-2 & 1.24e-2 & 8.30e-3 & 6.23e-3 & 4.99e-3 \\
                  \textbf{1.0, $\pmb{\pi}$/2} & 4.86e-2 & 2.47e-2 & 1.24e-2 & 8.30e-3 & 6.23e-3 & 4.99e-3 \\ \hline
               \end{tabularx}
            \end{center}
            \caption{The maximal relative errors of the approximation of Meixner-Pollaczek weight via $D_n^1$.}
         \end{table}
         
         \begin{figure}
           \label{img:2}
           \caption{The left-hand side: the graph of Meixner-Pollaczek weight for $\lambda=1/2, \varphi=\pi/3$. The right-hand side: the graph of the relative error of the approximation by $D^1_{60}$.}
           \centering
             \includegraphics[width=0.49\textwidth]{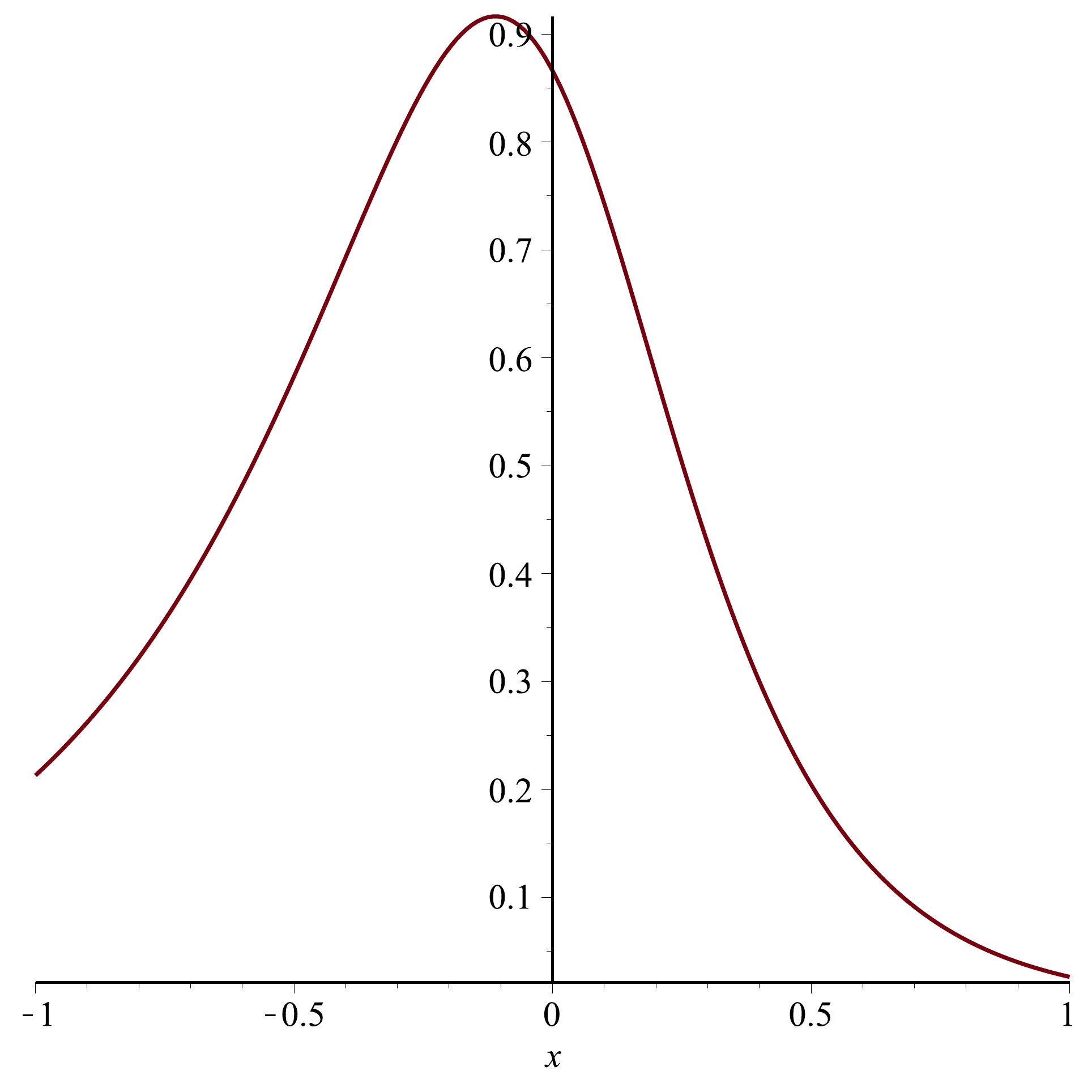}
             \includegraphics[width=0.49\textwidth]{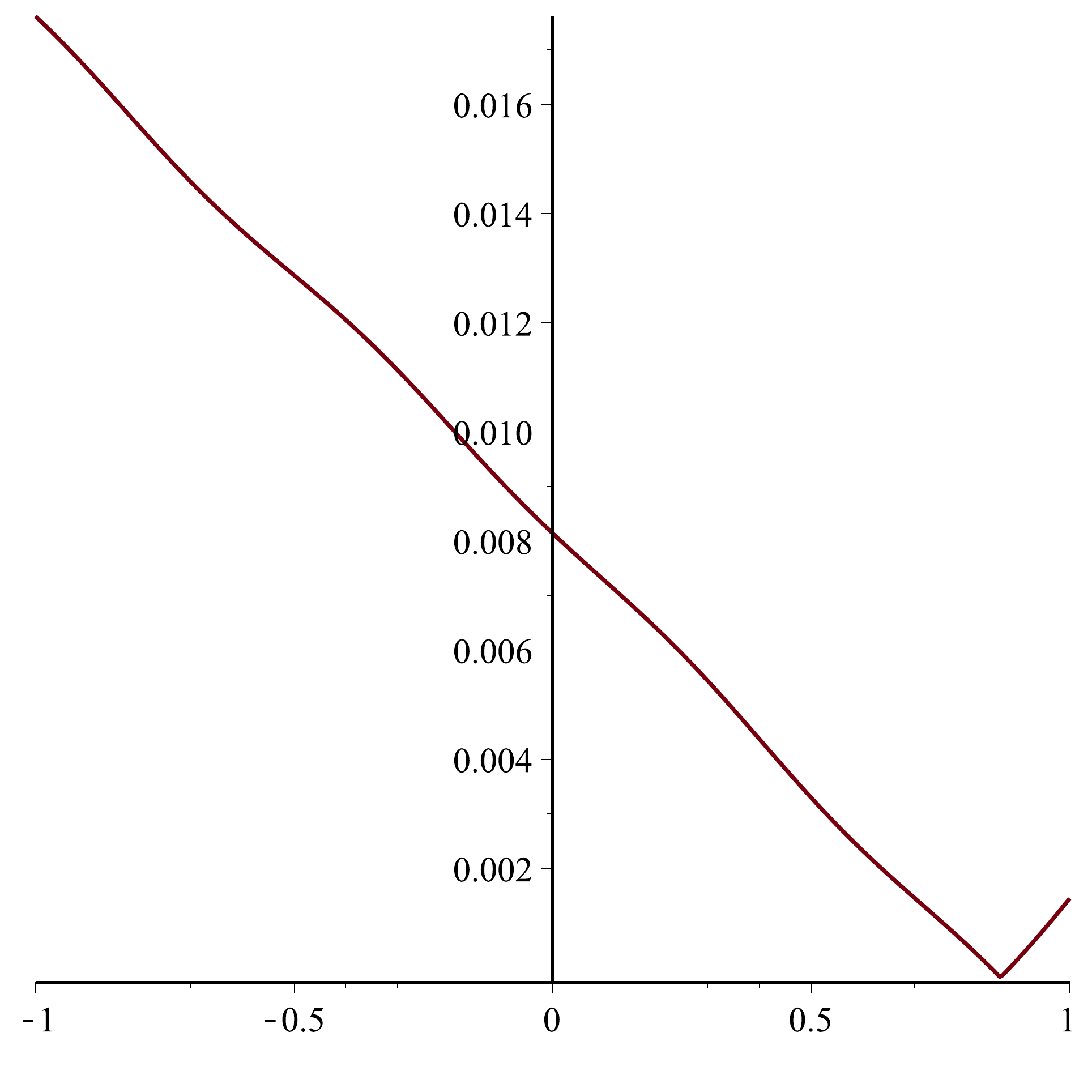}
         \end{figure}
      \end{przyklad}
      
      \begin{przyklad} \label{ex:3}
         Let $p_n$ be a sequence of \emph{Freud polynomials}, i.e. orthonormal polynomials associated with the measure
         \[
            \mu'(x) = c \ue^{-|x|^\beta}, \quad \beta \geq 1
         \]
         and $c$ is such that $\mu$ is a probability measure. Then according to \cite[Theorem 1.3]{Deift200147}, one has
         \[
            b_n \equiv 0, \quad \frac{a_n}{\widetilde{a}_n} = c' + r_n, \quad \widetilde{a}_n = (n+1)^{1/\beta}, \quad (r_n) \in \ell^1, \quad c' = \frac{1}{2} \left[ \frac{\Gamma(\beta/2) \Gamma(1/2)}{\Gamma((\beta+1)/2)} \right]^{1/\beta}.
         \]
         Observe that it implies
         \begin{equation} \label{eq:Freud1}
            \calV_1 \bigg(\frac{a_n}{\widetilde{a}_n} : n \in \NN \bigg) < \infty.
         \end{equation}
         Moreover, because $c' > 0$, it implies
         \begin{equation} \label{eq:Freud2}
            \calV_1 \bigg(\frac{\widetilde{a}_n}{a_n} : n \in \NN \bigg) < \infty.
         \end{equation}
         Observe that the sequence $\widetilde{a}$ satisfies the assumptions of Corollary~\ref{tw:asymptReg} with $N=1$. Moreover,
         \[
            \frac{a_{n+1}}{a_n} = \frac{a_{n+1}}{\widetilde{a}_{n+1}} \frac{\widetilde{a}_{n+1}}{\widetilde{a}_n} \frac{\widetilde{a}_n}{a_n}, \qquad \frac{1}{a_n} = \frac{\widetilde{a}_n}{a_n} \frac{1}{\widetilde{a}_n},
         \]
         and by \eqref{eq:Freud1} and \eqref{eq:Freud2}, each term has bounded total $1$-variation. Therefore, by \eqref{eq:29}, the sequence $a$ satisfies the assumptions of Corollary~\ref{tw:asymptReg} as well. It implies
         \[
            \lim_{n \rightarrow \infty} \widetilde{a}_n [p_{n-1}^2(x) + p_n^2(x)] = \frac{1}{\pi c c'} \ue^{|x|^\beta}
         \]
         almost uniformly on $\RR$.
         
         Notice that only for $\beta \in 2 \NN$ the density of $\mu$ is $C^\infty(\mathbb{R})$. In the other cases, it is $C^{\lceil \beta \rceil - 1}(\mathbb{R})$. It shows that similar situation could happen in Theorem~\ref{tw:przypadekRegularny} as it was the case in Example~\ref{ex:1}.
      \end{przyklad}
      
   \section*{Acknowledgements}
      I would like to thank Ryszard Szwarc for turning my attention to \cite[Theorem 7.34]{Nevai79}, and to Bartosz Trojan for his helpful suggestions concerning the presentation of this article.
   
   \bibliographystyle{abbrv} 
   \bibliography{measure}
\end{document}